\documentclass{amsart}
\usepackage[utf8]{inputenc}
\usepackage{amssymb,bm,amsfonts, stmaryrd, amsmath, amsthm, enumerate, mathtools}

\usepackage{wrapfig,tikz, tikz-cd}
\usetikzlibrary{arrows}

\usepackage[style=numeric]{biblatex}
\usepackage{xurl}
\usepackage[colorlinks]{hyperref}
\hypersetup{
 colorlinks=true,
 linkcolor=blue,
 filecolor=magenta, 
 urlcolor=cyan,
 breaklinks=true,
}
\usepackage[capitalise]{cleveref}
\crefformat{equation}{(#2#1#3)}
\crefrangeformat{equation}{(#3#1#4--#5#2#6)}
\crefformat{enumi}{(#2#1#3)}
\crefrangeformat{enumi}{(#3#1#4--#5#2#6)}

\newtheorem{theorem}{Theorem}[section]
\newtheorem{lemma}[theorem]{Lemma}
\newtheorem{proposition}[theorem]{Proposition}
\newtheorem{corollary}[theorem]{Corollary}

\newcounter{intro}
\newtheorem{introthm}[intro]{Theorem}

\theoremstyle{definition}
\newtheorem{definition}[theorem]{Definition}
\newtheorem{example}[theorem]{Example}
\newtheorem{remark}[theorem]{Remark}

\newtheorem{chunk}[theorem]{}

\usepackage{comment}
\newif\ifshowcomments
\newcommand{\michael}[1]{\ifshowcomments\textcolor{purple}{(#1)}\fi}
\newcommand{\cb}{\color{blue}}

\newcommand{\Hom}{{\operatorname{Hom}}}

\newcommand{\Tor}{{\operatorname{Tor}}}

\newcommand{\del}{\partial}
\renewcommand{\H}{\operatorname{H}}

\newcommand{\xra}{\xrightarrow}

\newcommand{\vp}{\varphi}

\newcommand{\ann}{{\operatorname{ann}}}
\newcommand{\soc}{{\operatorname{soc}}}

\newcommand{\shift}{{\scriptstyle\mathsf{\Sigma}}}

\renewcommand{\-}{\text{-}}

\newcommand{\m}{\mathfrak{m}}
\newcommand{\n}{\mathfrak{n}}

\newcommand{\N}{\mathbb{N}}
\newcommand{\Z}{\mathbb{Z}}
\DeclareMathOperator{\im}{im}

\newcommand{\isom}{\cong}

\DeclareMathOperator{\depth}{depth}
\DeclareMathOperator{\BI}{BI}
\DeclareMathOperator{\Burch}{Burch}
\newcommand{\krank}{k\-\operatorname{rank}}
\newcommand{\rank}{\operatorname{rank}}
\DeclareMathOperator{\syz}{syz}
\DeclareMathOperator{\coker}{coker}

\DeclareMathOperator{\edim}{edim}

\author[M.~DeBellevue]{Michael DeBellevue}
\address{Mathematics Department, Syracuse University, Syracuse, NY 13244 U.S.A.}
\email{mpdebell@syr.edu}

\author[C.~Miller]{Claudia Miller}
\address{Mathematics Department, Syracuse University, Syracuse, NY 13244 U.S.A.}
\email{clamille@syr.edu}

\date{\today}

\title[$k$ Summands Via Canonical Resolutions]{$k$ Summands of Syzygies over Rings of Positive Burch Index Via Canonical Resolutions}

\DeclareFieldFormat{postnote}{#1}
\DeclareFieldFormat{multipostnote}{#1}

\showcommentstrue
\begin{document}

\maketitle

\begin{abstract}
In recent work, Dao and Eisenbud define the notion of a Burch 
index, expanding the notion of Burch rings of Dao, Kobayashi, and Takahashi, and show that for any module over a ring of Burch index at least 2, its 
$n$th syzygy contains direct summands of the residue field for $n=4$ or $5$ and all $n\geq 7$. We investigate how this behavior is explained by the bar 
resolution formed from appropriate differential graded (dg) resolutions, yielding a new proof that includes all $n\geq 5$, which is sharp. When the module is Golod, we use instead the bar resolution formed from $A_\infty$ resolutions to identify such $k$ summands explicitly for all $n\geq 4$ and show that the number of these grows exponentially as the homological degree increases.
\end{abstract}

\section*{Introduction}
\label{sec-intro}

The results in this paper focus on finding copies of the residue field of a local or graded-local ring appearing as direct summands of syzygies of any finitely generated module over certain rings. This continues the work of Dao and Eisenbud in \cite{DaoEis-23}, in which they find that this phenomenon occurs surprisingly frequently. 
In the first part of the paper, we give a different proof of their result which explains the homological persistance of summands of syzygies as a manifestation of the self-similarity of the bar complex.
The key is identifying certain low degree non-Koszul cycles in all resolutions of the ring over an ambient regular ring.
Applied to resolutions equipped with a differential graded algebra and differential graded module structure, we obtain cycles in the bar resolution of the module over the ring of interest.
This shows that simple summands of syzygies are caused by the beginning of the resolution of the ring, which explains the ubiquity of this phenomenon.

These cycles spring from positivity of a numerical invariant of a local ring $(R,\m,k)$ called the Burch index, denoted by $\Burch(R)$.
This was introduced by Dao and Eisenbud to refine the definition of a Burch ring, first defined by Dao, Kobayashi, and Takahashi in \cite{DaoKobTak-20} as a ring with positive Burch index.
When $R$ is a Burch ring, resolutions of $\hat{R}$ over a regular ring $Q$ have strong linearity properties \cite[2.4]{DaoEis-23} which dominate the resolutions of all finitely generated $R$-modules \cite[3.1,~3.7]{DaoEis-23}.
Rings with larger Burch index are plentiful; in particular, \cite[Example~4.7,~Proposition~1.5]{DaoEis-23} provide large infinite classes of such rings.

Larger Burch index provides stronger results: When $R$ is a depth zero ring with Burch index at least $2$, Dao and Eisenbud prove that $\syz_i^R(M)$ has a $k$ summand for some $i\in\{4,5\}$ and all $i\geq 7$ and give an example to show that it need not occur for $i=4$~\cite[3.1, 3.6]{DaoEis-23}.
We give a new proof which, as a practical consequence, shows that $k$ summands are found in every degree five and beyond, giving a more uniform variant of \cite[Thm~3.1]{DaoEis-23} as follows.
This result is a corollary of the result established in \cref{thm:general} via \cref{dfn-burch-absolute}. 

\begin{introthm}\label{introthm:general}
    Let $(R,\m,k)$ be a local (or graded-local) ring of depth $0$ with  $\Burch(R)\geq 2$, and let $M$ be a finitely generated (resp., finitely generated graded) $R$-module. 
    Then $\syz_i^R(M)$ has a $k$ summand for each $i\geq 5$.
\end{introthm}

In the second part of the paper, we refine our technique by using $A_\infty$ structures on minimal resolutions of the ring and module over an ambient ring.
These structures are sufficient to construct the bar resolution, which is minimal when $M$ is Golod \cite{Bur-15}.
Under this additional assumption, we build cycles in the bar resolution which precisely correspond to syzygy summands. 
This establishes an explicit exponential lower bound in the number of such summands in  homological degree at least $4$.
In analogy with free rank, we let $\krank(M)$ be the number of independent $k$ summands of a module $M$.
This result is a corollary of the result established in \cref{thm:golod} via \cref{dfn-burch-absolute}. 

\begin{introthm}\label{introthm:golod}
Let $(R,\m,k)$ be a local (or graded-local) ring with $\Burch(R)=b\geq 2$, and let $M$ be a finitely generated (resp., finitely generated graded) Golod $R$-module.
Suppose that $R$ or its completion $\hat{R}$ are presented minimally as $Q/I$, where $(Q,\n,k)$ is regular local (or graded-local) ring and $I\subseteq \n^2$.
Then for each $i\geq 4$
\[
\krank({\syz_i^RM})\geq {b\choose 2} (\mu(I))^{\lfloor{\frac{i-4}{2}}\rfloor}.
\]
where $\mu(I)$ denotes the minimal number of generators of $I$. 
\end{introthm}

Note that, even when $R$ is not Golod, if $\Burch(R)\geq 2$, it already follows from \cite{DaoEis-23} that a subsequence of $\krank(\syz_R^i(M))$ grows exponentially; see \cref{rmk:exponential}. 
Here, however, noting that $\mu(I)\geq 2$ when $\Burch(R)\geq 2$, we see exponentially-many explicit cycles in all degrees $q\geq 4$ providing these summands. 

In their paper, Dao and Eisenbud give an example to show that Burch index 1 does not suffice for the existence of $k$ summands in syzygies.
In \cref{chu:exampleBIone}, we explore this example to illustrate where in our methods a Burch index of at least 2 is crucial. 

The paper begins with a background section, \cref{sec-background}, with definitions of $\Burch(R)$ and the relative version $\Burch(\varphi)$ for a ring map $\varphi$. 
We also remind the reader of $A_\infty$ and differential graded algebra structures and their use in constructing the bar resolution.
In the preliminaries section, \cref{sec-preliminaries}, we define a certain set of cycles in the resolution of $R$ over a regular ring which we call \emph{Burch cycles}.
We then use these cycles and the bar resolution to produce $k$ summands in syzygies when the Burch index is at least 2, first in a general setting in \cref{sec-general} and then, with more precision, in the relative Golod setting in \cref{sec-golod}. 

One last remark: While the theorems above were stated in terms of a fixed ring relative to an ambient ring given by a minimal Cohen presentation, we prove more general results about ring homomorphisms.
They are related to the statements in the introduction in \cref{dfn-burch-absolute}.

\section{Background}
\label{sec-background}
\begin{chunk}\label{chu-presentation}
Let $(Q,\n)$ and $(R,\m)$ be local commutative Noetherian rings with a fixed surjective homomorphism $\vp\colon Q\rightarrow R$ and common residue field $k$ and kernel $I$. 
An important application is when $\vp$ is the minimal Cohen presentation of a complete or positively graded ring. 
Modules $M$ over $R$ or $Q$ are always assumed to be finitely generated.
\end{chunk}

\begin{chunk}\label{chu-graded}
The results in this work apply equally well to graded-local rings with a graded-local homomorphism $\vp\colon (Q,\n)\rightarrow (R,\m)$, by which we mean that $Q$ and $R$ are $\N$-graded (or $\N^l$ multigraded) rings whose degree $0$ pieces are a common field $k$.
In this setting, we further require that modules be graded also.
\end{chunk}



\begin{definition}\label{dfn-burch}
Here we recall from \cite{DaoEis-23} the definitions of the Burch ideal and the Burch index of $\vp$. 
The \emph{Burch ideal} of the map $\vp$ with kernel $I$ is defined to be 
\[
\BI(\vp)= I\n : (I : \n). 
\]
Note that when $\depth R=0$, $\soc(R)\neq0$, and the preimage of $\soc(R)$ in $Q$ is the ideal $(I : \n)$. 
By abuse of language, we may call elements of this ideal {\it socle elements} in descriptive text.

Define the \emph{Burch index} of $\vp$ to be 
\[
\Burch(\vp)=\dim_k \frac{\n}{\BI(\vp)}
=\dim_k \frac{\n}{(I\n : (I : \n))}
\]
When $\depth(R)\neq 0$ or $I=0$, $(I : \n)=I$, and so $\BI(\vp)=\n$ and $\Burch(\vp)=0$.

For a more qualitative description, note the following: The ideal $I\n$ records those elements of $I$ which are not minimal generators, so the iterated colon ideal $I\n : (I : \n)$ describe lifts of elements of $R$ which multiply the lift of the socle outside of the set of minimal generators of $I$.
Evidently, $\n^2$ is contained in this set: $\n (I : \n)\subseteq I$ by the definition of the colon ideal, and then multiplying on the left by $\n$ gives $\n^2 (I : \n) \subseteq \n I$, whence $\n^2\subseteq \BI(\vp)$.
Therefore, there is a surjection  
\[
\n/\n^2 \twoheadrightarrow \n/\BI(\vp)
\]
The latter is the $k$-space of generators of $\n$ which multiply some minimal generator of the lift of the socle to a minimal generator of $I$. Therefore, for each $x\in \n/\BI(\vp)$, there are corresponding socle elements $s$ whose images in $(I : \n)/\n (I : \n)$ are nonzero and which satisfy $xs\in I/\n I$. 

This yields an \emph{equivalent description} of the Burch index which we will use later: Setting $b=\Burch(\vp)$, there is a minimal generating set $\{a_1,\dots,a_m\}$ of $I$ with the property that there exist $x_1,\dots,x_b \in \n$ whose images are \emph{independent} in the $k$-vector space $\n/\BI(\vp)$ (and hence also in $\n/\n^2$) such that for $i=1,\dots,b$
one has 
\[
a_{j_i}=x_is_i 
{\textrm{ for some }} 
s_i\in(\n:I) 
{\textrm{ and some }} 
j_i.
\]
Note that we may have $a_{j_i}=a_{j_\ell}$ for some $i\neq \ell$.


\end{definition}

\begin{definition}\label{dfn:krank}
Let $M$ be a finitely generated module over a local ring $(R,\m,k)$. 
In analogy with the more traditionally studied free rank, we set the \emph{$k$-rank of $M$} to be the maximal number of direct summands of the residue field:
\[\krank_R(M)=\max\{n\,|\, k^n\xhookrightarrow{\oplus} M\}\]
where we use the notation $N\xhookrightarrow{\oplus}M$ to denote that $N$ is a direct summand of $M$.
We write $\krank(M)=\krank_R(M)$ when $R$ is clear from context.
The use of maximum rather than supremum is justified by the fact that $M$ is finitely generated: any such map remains injective upon tensoring with $R/\m$.
We say that an element $m\in M$ generates a direct $k$-summand if $km\xhookrightarrow{\oplus}M$.
\end{definition}

\begin{proposition}\label{prop:completion} 
If $M$ is finitely-generated, then $\krank(M)$ is invariant under completion.
\end{proposition}
\begin{proof}
If $\krank_{\hat R}(\hat M) = d$, then \[\soc_R(M)=\widehat{\soc_R(M)}=\widehat{\Hom_R(k,M)}\isom \Hom_{\hat R} (\hat k,\hat M)=\soc_{\hat R}(\hat M),\]
and this isomorphism respects the natural inclusion $M \to \hat M$. 
The splitting $\hat M \rightarrow \soc_{\hat R}(\hat M) \cong k^d$ then factors the identity of $k^d$ as $k^d\cong \soc_R(M) \rightarrow M\rightarrow \hat M\rightarrow k^d$.
\end{proof}

\begin{chunk}\label{dfn-burch-absolute}
In \cite{DaoEis-23} the \emph{absolute Burch index}, $\Burch(R)$, of the ring $R$ is also defined by taking $\vp$ to be a minimal Cohen presentation, but we will not make use of this notion in this work.  
However, to translate our results to the special cases presented as Theorems A and B in the introduction, we describe it briefly: 
By the Cohen Structure Theorem, one may write the completion of $R$ as the quotient $\hat R=Q/I$ of a regular local (or graded-local) ring $Q$ by an ideal $I$ in the square of the (homogeneous) maximal ideal. 
Let $\vp\colon Q \to \hat R$ be the natural surjection. 
Then one has 
\[
\Burch(R) = \Burch(\vp) = \dim_k \frac{\n}{(I\n : (I : \n))}
\]
\cref{thm:general,thm:golod} discussed below then extend to the absolute case by \cref{prop:completion}.
\end{chunk}
\begin{chunk}\label{chu-dga}
Bar resolutions are an essential tool used in this work. 
Their construction involves building projective resolutions over $R$ from those over $Q$.
The resolutions over $Q$ must be equipped with $A_\infty$ structures.
The simplest and most classical of these is that of a \emph{differential graded (dg) algebra}, first introduced into commutative algebra by Tate~\cite{Tate:1957}.
We recall their definition below:

\begin{definition} A \emph{dg algebra} $X$ over $Q$ is complex of $Q$-modules further equipped with a chain map $m_2\colon X\otimes_Q X\rightarrow X$  (the $Q$-bilinear product); writing the product as concatenation in the ordinary way, that this product is a chain map is encoded by the Leibniz rule:
\[\del(ab)=\del(a)b+(-1)^{|a|}a\del(b){\textrm{ for all (homogeneous) }}a,b\in X\]
where $|a|$ denotes the degree of $a\in X$; note that by our convention, $X$ is a complex and so elements are automatically homogeneous.
Denoting the differential by $m_1$ and the product map by $m_2$, this can be restated as
\looseness -1
\[m_1(m_2(a,b))=m_2(m_1(a),b)+(-1)^{|a|}m_2(a,m_1(b)){\textrm{ for all }}a,b\in X.\]
We require that $\oplus_n X_n$ is a $Q$-algebra in the ordinary sense once the differential has been forgotten: in particular, $m_2$ must be associative and $X$ is equipped with a unit map $u\colon Q\rightarrow X$.

A \emph{dg module} $Y$ over a dg algebra $X$ is a complex $Y$ of $Q$-modules further equipped with a chain map $X\otimes_Q Y\rightarrow Y$ which is compatible with differentials of $X$ and $Y$, as encoded by an analogous Leibniz rule.

A \emph{dg algebra resolution} of $R$ over $Q$ is a dg algebra $X$ with $H(X)\isom R$, each $X_n$ a projective $Q$-module, and $X_{<0}=0$.
For an $R$-module $M$, a \emph{dg module resolution} of $M$ over $X$ is a dg module $Y$ with $H(Y)\isom M$, each $Y_n$ projective over $Q$, and $Y_{<0}=0$.
\end{definition}

\begin{chunk}\label{c:dga-existence}
While the minimal $Q$-free resolutions of $R$ and $M$ need not support dg structures, there always exists \emph{some} resolution $X$ of $R$ with a dg algebra structure \cite{Tate:1957} and \emph{some} resolution $Y$ of $M$ with a dg $X$-module.
When $\vp\colon Q\rightarrow R$ is surjective and $M$ is a finite $R$-module, then such resolutions exist with
\begin{enumerate}
    \item $X_0=Q$, and each of $X_n$ and $Y_n$ is finite over $Q$ for all $n$,
    \item forgetting differentials, the underlying algebra structure of $X$ is that of a free graded-commutative $Q$-algebra \cite[2.1.10]{Avramov:1999}, and
    \item forgetting differentials, the underlying module structure of $Y$ over $X$ is free \cite[2.2.7]{Avramov:1999}.
\end{enumerate}
\end{chunk}

\end{chunk}

\begin{definition}\label{def-Ainfinity}
    Let $X$ be any resolution of $R$ over $Q$, and $Y$ be any $Q$-free resolution of an $R$-module $M$.
    Any lifts of the multiplication maps $Q\otimes_Q Q\rightarrow Q$ and $Q\otimes_Q M\rightarrow M$ provide candidate multiplication maps $m_2\colon  X\otimes_Q  X\rightarrow X$ and $\mu_2\colon  X\otimes_Q Y\rightarrow Y$ for dg algebra and module structures of $X$ and $Y$, respectively.
    Such lifts may fail to be associative.
    For some rings and modules, when $X$ and $Y$ are minimal resolutions over $Q$, no such lifts are associative~\cite{Avramov:1981}.
    To rectify this, we make use of an $A_\infty$ structure, the needed details of which we recall below; see~\cite{LodVal-12, Kel-06,Lef-02} for more of the general theory and \cite{sta-61,sta-19} for their history and origins.
   
An \emph{$A_\infty$-algebra} over a ring $Q$ is a complex $A$ of $Q$-modules together with $Q$-multilinear maps of degree $n-2$ 
\[
m_n \colon A^{\otimes n} \to A, {\textrm{ \ for }}n\geq 1
\] 
called operations or multiplications, satisfying the Stasheff identities for each $n\geq 1$:
\[
\sum_{s=1}^n
\sum_{\substack{r+s+t=n\\ r,t\ge 0}}  (-1)^{r+st} 
m_{r+1+t} (1^{\otimes r} \otimes m_s \otimes 1^{\otimes t}) = 0
\]
where the elements are automatically homogeneous as the underlying structure of $A$ is that of a complex.
The \emph{associator} of the multiplication $m_2$ sends $a\otimes b\otimes c$ to $a(bc)-(ab)c$ (where we write $ab=m_2(a\otimes b)$).
The third Stasheff identity is \begin{align*}
m_2(&1 \otimes m_2 - m_2 \otimes 1)  \\ &=m_1m_3 +m_3(m_1 \otimes 1 \otimes 1+1 \otimes m_1 \otimes 1+1 \otimes 1 \otimes m_1)
\end{align*}
which shows that the associator is the boundary of $m_3$ in 
$\Hom_R(A^{\otimes 3},A)$.
In other words, the associator is \emph{nullhomotopic} with nullhomotopy $m_3$.

Note that when one applies the maps in each formula above 
to an element, one should use the Koszul sign rule: 
For graded maps $f$ and $g$, one has 
\[
(f \otimes g)(x \otimes y) = (-1)^{|g||x|}f(x) \otimes g(y)
\]
where $|g|$ denotes the degree of $g$.

We also require $A_\infty$ algebras to be \emph{strictly unital}, in that there is a unit map $u\colon Q\rightarrow A$ such that for $r+t=n\neq 2$, $m_n(1^{\otimes r}\otimes u\otimes 1^{\otimes t})=0$.
Intuitively, this just states that multiplication by the unit is associative~\cite[7.2]{Positselski:11}.

An \emph{$A_\infty$-module} $M$ over an $A_\infty$-algebra $A$ is a complex $M$ of $Q$-modules together with $Q$-multilinear maps of degree $n-2$ 
\[
\mu_n \colon A^{\otimes {n-1}}\otimes M \to M, {\textrm{ \ for }}n\geq 1
\] 
such that analogous Stasheff identities hold. 
\end{definition}

In fact, $A_\infty$ are more affordable than dg algebra structures, as manifested by the following:

\begin{proposition}[\cite{Bur-15}, Proposition 3.6]\label{prop:burkeAinfty}
    Let $X$ be any resolution of $R$ over $Q$ and $Y$ be any $Q$-resolution of an $R$-module $M$.
    Then there exists an $A_\infty$ algebra structure on $X$ and an $A_\infty$ $X$-module structure on $Y$.
\end{proposition}

\begin{chunk}\label{chu-bar}
Our essential technique to find summands of $k$ in syzygies is to use the \emph{relative} bar resolution, first developed by Iyengar for dg structures \cite{Iye-97}. 
In this case, the relative bar resolution arises as the totalization of the classical bar complex~\cite{Eilenberg/MacLane:1951}.
We use the relative bar resolution in this form to obtain \cref{thm:general}.
For \cref{thm:golod}, it is essential that the input data be minimal resolutions, and so we must make use of a generalization of the relative bar resolution to $A_\infty$ structures.
This generalization was developed by Burke \cite{Bur-15} and Positselski \cite{Positselski:11}.
We recall the constructions here. 

Let $X\simeq_Q R$ be a $Q$-free resolution with an $A_\infty$ algebra structure, and let $Y\simeq_Q M$ be a $Q$-free resolution with an $A_\infty$ $X$-module structure.
Let $u\colon Q\rightarrow X$ be the unit map, and set $\overline{X}=\shift\coker u$, where $\shift F$ of a complex $F$ is the complex with $\shift F_n = F_{n-1}$ and differential $-\del_F$. 
In this paper, since $Q\rightarrow R$ is surjective, we may identify $\overline X$ with $\shift (X_\geq 1)$.

 For each $n$, define \[B_n(R,X,Y)=\bigoplus\limits_{i_1+\dots+i_p+j+p=n}R\otimes_Q\overline X_{i_1}\otimes_Q\dots\otimes_Q\overline X_{i_p}\otimes Y_{j}.\]
As is traditional for the bar complex, we denote the element $r \otimes x_1 \otimes x_2 \otimes \dots \otimes x_p \otimes y$ of $B_n(R,X,Y)$ by 
$$
r[x_1|x_2|\dots|x_p]y
$$
which is said to be of \emph{bar length} $p$.
The differential is the signed sum of all $A_\infty$ operations applied to every substring of appropriate bar length:

\begin{align*}
\del(r[x_1|\dots|x_p]y)
&=
\sum_{i=1}^p \sum_{j=0}^{p-i} 
\pm\,
r[x_1|\dots|x_j|m_i(x_{j+1}\otimes\dots\otimes x_{j+i})|x_{j+i+1}|\dots|x_p]y
\\
&+
\sum_{i=1}^{p+1}
\pm\,
r[x_1|\dots|x_{p-i+1}]\mu_i(x_{p-i+2}\otimes \dots\otimes x_p\otimes y)
\end{align*}
where, for visual simplicity, we suppress the exact signs, which are irrelevant to our application.
The augmentation is given by
\[
\varepsilon(r[x_1|\dots|x_p]y)
=\begin{cases}
r\varepsilon^ Y(y) & {\textrm{ if }} p=0 \\
0 & {\textrm{ if }} p>0
\end{cases}
\]

In the dg setting, $m_{\geq 3}=\mu_{\geq 3}=0$ and so the differential simplifies as follows.
Here we do indicate the signs as we will need them.
Writing concatenation in place of $m_2$ and $\del$ in place of $m_1$ and $\mu_1$, the differential is:

\begin{align*}
    \sum\limits_{t=1}^{p}(-1)^{|x_1|+\cdots +|x_{t-1}|
    }&\, r[x_1|\dots|\del(x_t)|\dots|x_p]y
    \\
    &\hspace{-20pt}+(-1)^{|x_1|+\cdots + |x_p|}r[x_1|\dots x_p]\del(y)
    \\
    + \sum\limits_{t=1}^{p-1} (-1)^{|x_1|+\cdots + |x_{t-1}|}&\, r[x_1|\dots|x_tx_{t+1}|\dots|x_p]y
    \\
    &\hspace{-20pt}+(-1)^{|x_1|+\cdots + |x_{p-1}|}r[x_1|\dots| x_{p-1}]x_py 
    \\
\end{align*}
where $|x_i|$ refers to the degree of $x_i$ in the shifted complex $\overline X$.
Note that the formula for $d$ does not have a term involving multiplication by the factor of $R$; heuristically, this is because the product $R \cdot \overline{X}$ vanishes as $R=\H_0(X)$. 
\end{chunk}
    
\section{Preliminaries: Burch cycles}
\label{sec-preliminaries}

In this section, given a local (or graded-local) homomorphism $\vp\colon Q \to R$, we find some distinguished cycles in any $Q$-free resolution of $R$ when the Burch index of $\vp$ is positive. For ease, we work in a minimal resolution but the resulting cycles therefore exist in any resolution. 

\begin{chunk}\label{chu-fij}
Let $X$ be the minimal resolution of $R$ over $Q$, and let $e_1,\dots,e_m$ be a basis of $X_1$,  $a_1,\dots, a_m$ a minimal generating set of $I$, and $\del(e_i)=a_i$.
When $\Burch(R)>0$, one or more of the generators $a_{j_i}$ are multiples of socle elements as described in \ref{dfn-burch}.
With $a_{j_i}=x_is_i$ for some elements $x_i\in \n\backslash\BI(\vp) \subseteq \n\backslash \n^2$ and element $s_i\in (I : \n)$, take any $x\in \n\setminus\n^2$ whose image in $\n/\n^2$ is not in  the span of the image of $x_i$. Then, since $xs_i$ is in $I$, one may write 
\[
\del(xe_i)=x(x_is_i)=x_i(xs_i)
=x_i\left(\sum_\ell r_\ell^{xx_i} a_\ell \right)
=\del\left( x_i\left(\sum_\ell r_\ell^{xx_i} e_\ell \right)\right)
\]
for some elements $r_\ell^{xx_i}\in Q$ and so the element 
\[
\omega_{xx_i}=xe_i-x_i\left(\sum_\ell r_\ell^{xx_i} e_\ell \right)
\]
is a cycle in $X_1$, hence a boundary, so fix $f_{x,x_i}\in X_2$ with 
\[
\del(f_{x,x_i})=\omega_{xx_i}. 
\]

Note that the conditions on $x$ ensure that each $\omega_{xx_i}$ is a minimal generator of the set of degree one cycles $Z_1(X)$: Since $\del^X_1 \colon X_1\rightarrow I$ identified a minimal generating set, one has $Z_1(X)\subset \n X_1$, so $\n Z_1(X)\subset \n^2 X_1$.
As $x$ is independent from $x_i$, the image $\omega_{x,x_i}$ is non-zero in $\n X_1/\n^2 X_1$, and hence $\omega_{x,x_i}$ is nonzero in $\Z_1(X)/\n Z_1(X)$.  
Therefore, the chosen preimage $f_{x,x_i}$ is part of a basis of $X_2$.

The above is extended to all of $\n/\BI(\vp)$ by the following lemma. 
Note that the obvious choices of $r_\ell$ in the construction above yield $\omega_{x_i,x_j}=-\omega_{x_j,x_i}$ and $\omega_{x_i,x_i}=0$; hence we only include those for $i<j$:

\begin{lemma}\label{lem-independent}
Let $\overline{x_1},\dots,\overline{x_b}$ be a basis of $\n/\BI(I)$, and extend it to a basis $\{\overline{x_1},\dots,\overline{x_n}\}$ of $\n/\n^2$.
The image of the set
\[\{\omega_{x_j,x_i}=x_je_i-x_i\left(\sum_\ell r_\ell^{x_j,x_i} e_\ell \right)\,|\,i<j,\,1\leq i\leq b\}\]
in $Z(X_1)/\n Z(X_1)$ is linearly independent, and hence in any completion to a resolution $X$, there are basis elements $f_{x_j,x_i}\in X_2$ mapping to each corresponding $\omega_{x_j,x_i}$.

Furthermore, for each $i,j$, we have that $\del(f_{x_j,x_i})\notin BI(I)X_1$.
\end{lemma}
\begin{proof}
    Order the basis $\{x_je_i\}$ of $\n X_1/\n^2 X_1$ lexicographically. The set $\{\omega_{x_jx_i}\}$ is  obtained from a subset of the basis $\{x_je_i\}$ by subtracting a linear combination of basis elements that are strictly smaller in the ordering, and so they are independent. 

    Note that in the sum, when $\ell=i$ the image of the coefficient $x_j-x_i r_i^{x_j,x_i}$ of $e_i$ in $\n X_1/\BI(I) X_1$ is non-zero since $x_j$ and $x_i$ are linearly independent in $\n/\n^2$ and $x_j\not\in\BI(I)$. 
\end{proof}
\end{chunk}

Once we build cycles in the bar complex from these distinguished cycles in $X$ that generate copies of $k$ in the syzygies of a module, we show that they split out of the syzygies with the following modified version of \cite[Proposition 3.4]{DaoEis-23}. 

\begin{lemma}[\cite{DaoEis-23}]\label{lem-splitting}
Let $\vp\colon (Q,\n,k)\rightarrow (R,\m, k)$ be a surjective local homomorphism.
Let $\del\colon F\rightarrow G$ be a map of free $Q$-modules and $\rho\in F$ be nonzero in $F/\n F$.
Suppose that $\del(\rho)\in \n G$ but $\del(\rho)\notin \BI(\vp)G$, equivalently, there exists an $s\in (I : \n)$ with 
\[
s \del(\rho)\in I G{\textrm{ but }}s \del(\rho)\notin \n I G.
\]
Then the element $\bar{s}\rho$ generates a direct summand of $\ker(R\otimes_Q \del)$ isomorphic to $k$.
\end{lemma}

Note that the conditions of the lemma are equivalent to $R\otimes \del(s\rho)=0$ for any $s\in (I:\n)$ chosen so that $\del(s\rho)\notin \n I$, which is what we will check when we apply the lemma.


\begin{chunk}\label{chu:exampleBIone}
As an illustration for what goes wrong with our arguments for \cref{thm:general,thm:golod} if a map has only Burch index at most 1, here we review the example of Dao and Eisenbud of a ring of absolute Burch index 1 and a module with no syzygies containing a summand isomorphic to $k$ \cite[3.4]{DaoEis-23}. We show that the cycles we built do generate modules isomorphic to $k$, but that we cannot conclude they are direct summands.

With $\vp\colon Q=k[x,y] \to R=k[x,y]/(x^2,y)^2=k[x,y]/(x^4,x^2y,y^2)$ and $M=R/(x^2,y)$, we have for $I=\ker\vp=(x^2,y)^2$ that $\BI(I)=(x^2,y)$, so $\n/\BI(I)=kx$.
Then $x^4=x\cdot x^3=xs_2$ and $x^2y=x\cdot xy=xs_1$ are the only ways to write the minimal generators $\{x^4,x^2y,y^2\}$ as elements of $\n (I:\n)$.

In the minimal resolution $X$ of $R$, as in \cref{chu-fij}, we have
$\del(e_1)=x^4$ and $\del(e_2)=x^2y$, and the corresponding Burch cycle is $ye_{2}-x^2e_{1}$.
Recall we denote a preimage of this cycle by $f_{y,x}\in X_2\setminus \n X_2$.
To employ Dao-Eisenbud's lemma to elements in the bar resolution involving this cycle, one would need that $\del(sf)\notin \m IX$ for some socle element $s\in (I:\n)=(x^3,xy,y^2)$. 
As the coefficients of $\del f$ are contained in $\BI(I)=(y,x^2)$, we have that $\del(sf)=s(ye_{2}-x^2e_{1})\in\n IX$ for all $s\in (I:\n)$.

This explains why our set of distinguished Burch cycles $f_{x_i,x_j}$ always involves choosing two elements $x_i$, $x_j$ whose images are independent in $\n/\BI(I)$ and hence we require Burch index at least 2.
\end{chunk}

\section{General case}
\label{sec-general}
Let $\vp\colon Q \to R$ be a local (or graded-local) homomorphism of rings, and let $M$ be a finitely generated $R$-module. 
Set $b$ to be the Burch index of $\vp$. 
In this section, we tackle the general case of any finitely generated $R$-module when $\Burch(\vp)\geq 2$. 
We work first to find summands of cycles in a bar resolution constructed from nonminimal resolutions with dg structures, and show that this induces a summand of each sufficiently high syzygy of $M$. 
We begin by describing the details of the resolutions we will use.

\begin{chunk}\label{chu-general-setup}
First, as in \cref{chu-dga}, take $X\xra{\simeq}_QR$ to be a dg algebra resolution of $R$ over $Q$ such that there is a basis $e_1,\dots,e_m$ of $X_1$ with $\del(e_i)=a_i$. 
As described in \cref{chu-fij}, there is an associated distinguished set of elements $\{f_{x_j,x_i}\,|\,i<j\} \subseteq X_2$ which the differential maps to the cycles constructed there and which are $Q$-linearly independent by \cref{lem-independent}.

Next, as in \cref{chu-dga}, we take $Y\xra{\simeq}_Q M$ to be a semifree dg $X$-module resolution $Y$ of $M$ but constructed with care so that there is dg module map $\psi\colon X\rightarrow Y$ that is degree-wise a split injection.
As described in \cite{Avramov:1999}, the first step is to take a $Q$-linear surjection $Q^n \to M$ and lift it to a dg $X$-module map $Y(0):=X^n \to M$ where $X^n=Q^n\otimes_Q X$; this is trivially possible as $\im \del^X_1=I\subseteq\ann_QM$. 
Continuing to kill cycles, one constructs a semifree dg $X$-module resolution $Y$ of $M$ with the property that there is an injection $Y(0) \hookrightarrow Y$. 
The composition 
\[
\psi \colon X\xra{\iota}X^n=Y(0)\hookrightarrow Y, 
\]
where $\iota$ is any injection, is an injective dg module map. 
The choices of images at each step are basis elements, so $\iota$ splits degree-wise.
\end{chunk}


This is elucidated in the following example.

\begin{example}\label{exa:structure}
    Let $Q=k[x]$, $R=k[x]/(x^2)$, and $M=k=k[x]/(x)$.
    Then $X$ can be taken to be the Koszul complex $Q\langle e\,|\,\del(e)=x^2\rangle$,
    and $Y(0)=X$, so $\psi$ can be defined by setting $\iota$ to be the identity.
\end{example}

We begin by finding $k$-summands in the (possibly non-minimal) bar resolution. 

\begin{lemma}\label{lem:general}
Let $\vp\colon Q \to R$ be a local or graded local homomorphism, and let $M$ be a finitely generated $R$-module. 
Suppose $\Burch(\vp)=b\geq 2$.

Let $X$ be a dg algebra resolution of $R$ over $Q$ as in \cref{c:dga-existence}, and let $Y$ be a semifree dg module resolution of $M$ over $X$ with a dg $X$-module map $\psi\colon X \to Y$ as constructed in \cref{chu-general-setup}. 
Let $B=B(R,X,Y)$ be their associated bar resolution of $M$ over $R$, and let $e\in X_1\setminus \n X_1$.
For $i<j$, set
\[
\rho_{i,j}
= 1[f_{x_j,x_i}|
\underbrace{e|\dots|e}_{\frac{q-4}{2}}
]\psi(e)
- 1[e f_{x_j,x_i}|
\underbrace{e|\dots|e}_{\frac{q-4}{2}}
]\psi(1)
\in B_{q}
\]
when $q\geq 4$ is even and
\[
\rho_{i,j}
= 1[e f_{x_j,x_i}|
\underbrace{e|\dots|e}_{
\frac{q-5}{2}}
]\psi(e)
\in B_{q}
\]
when $q\geq 5$ is odd.
Then $\alpha_{i,j}=s_i\rho_{i,j}$ span a nonzero $k$-vector space that is a direct summand of $\ker(\del^B_q)$.
\end{lemma}
\begin{proof}
    First we verify that the elements $\alpha_{i,j}$ are indeed cycles in $B$. 
Applying the portion of the differential of $B(R,X,Y)$ which comes from the differentials of $X$ and $Y$ yields 0 since $s_i\in (I:\m)$, and the resolutions are minimal on the elements involved in the expression, that is, one has 
$\del(e), \del(f_{x_j,x_i}) \in \m X$ 
and $\del(\psi(e))=\psi(\del(e)) \in \m Y$.
Note that the tensors are over $Q$, and so the element $s_i$ can be moved across them.

The remaining terms come from the bar portion, denoted by $d$, of the differential.  
Many terms vanish from the fact that $e^2=0$ and $e\psi(e)=\psi(e^2)=0$, and the terms of $d$ involve products of adjacent factors.
When $q$ is odd, this causes all terms of $d$ to vanish.
When $q$ is even, the only terms remaining are
\[
s_i[f_{x_j,x_i}e|\dots|e
]\psi(e)
- s_i[e f_{x_j,x_i}|e|\dots
]e\psi(1)
=0
\]
where the first term comes from the multiplication of the first two bar (interior) factors of the first term of $\alpha_{i,j}$,
the second term comes from the multiplication of the last two tensor factors of the second term of $\alpha_{i,j}$, and the sign comes from the fact that $ef_{x_j,x_i}$ and each shifted $e$ has even degree.

The elements $\rho_{i,j}$ are part of a basis of $B_q$.
Lifting the differential $\del\colon B_q\rightarrow B_{q-1}$ to $Q$ and $\rho_{i,j}$ to a basis element $\rho=\widetilde{\rho_{i,j}}$, the assumptions on $f_{x_j,x_i}$ ensure that $\del(\rho)\notin \BI(\vp)B_{q-1}$. 
Hence each of these cycles generate a $k$-summand of the $q$-th syzygy by \cref{lem-splitting}, which suffices to yield the result. 
\end{proof}

\begin{remark}
As we are not presuming $M$ is $\vp$-Golod, the $A_\infty$ bar resolution cannot be expected to be minimal, even when the starting resolutions are minimal.
In light of this, we elect to use simpler dg structures, rather than the $A_\infty$ structure used in \cref{thm:golod} \cref{chu-bar}.
Using dg structures simplified the analysis of the cycles in the proof of \cref{lem:general}.
\end{remark}

To prepare for projecting to a minimal resolution sitting inside the bar resolution $B$ in the lemma above, we prove the following.

\begin{lemma}\label{lem:summand-free}
Let $(R,\m,k)$ be a local or graded. If $R$ is not a field, then $k$ cannot be a direct summand of a free $R$-module $R^c$. 
\end{lemma}

\begin{proof}
If $k$ is a summand of $R^c$, then it is projective and hence free since $(R,\m,k)$ is local or graded \cite[1.5.15(d)]{Bruns/Herzog:1993}.
But $k$ is not free unless $R$ is a field.
\end{proof}

\begin{lemma}\label{lem:nonminimal-to-minimal}
Let $(R,\m,k)$ be a local or graded ring, Let $G$ be a free resolution of $M$, and let $G= F\oplus C$ be a decomposition of $G$ with $F$ a minimal resolution of $M$ and $C$ contractible.
Suppose $f+c$ generates a direct $k$-summand of $\ker(\del^G_i)$, with $f\in F_i$ and $c\in C_i$.
Then $f$ generates a direct $k$-summand of $\ker(\del^F_i)=\syz_i(M)$.
\end{lemma}
\begin{proof}
Let $\alpha\colon G\rightarrow k$ be the splitting sending $f+c$ to $1$.
Then the map $k\rightarrow C_i$ sending $1$ to $c$ is well-defined, since $c$ is killed by $\m$.
The post-composition with $\alpha$ must be zero, else $c$ would generate a direct $k$-summand of the free module $\ker(\del^G_i)$, contradicting \cref{lem:summand-free}.
Then since $\alpha$ is nonzero, $\alpha(f)$ must be a unit $u$.
Hence the map $k\rightarrow F_i$ sending $1$ to $f$ is split by $u^{-1}\alpha.$
\end{proof}

Applying the above lemmas to the 
(possibly non-minimal) bar resolution $B=B(R,X,Y)$ from \cref{lem:general}, we immediately obtain the main result of this section:

\begin{theorem}\label{thm:general}
Let $Q\xra{\vp} R$ be a local or graded local homomorphism, and let $M$ be a finitely generated $R$-module. 
Suppose $\Burch(\vp)\geq 2$.

Then for $q\geq 4$ 
\[
\krank({\syz_{q+1}^RM})\geq 1.
\]
\end{theorem}

\cref{lem:general} identifies a much larger $k$-vector space summand of a resolution of $M$ over $R$, but it is not clear how many of these summands become identified upon projecting to the minimal resolution.
Computations suggest the lower bound in \cref{thm:general} is far from sharp, and it is possible that sharper bounds may be identified in some cases.
We provide such a sharper bound in the Golod case, as described in the next section.
\section{Golod case}
\label{sec-golod}

For Golod modules $M$, there is a bar resolution $B(R,X,Y)$ that minimally resolves $M$.
In this section, we use the minimality of the resulting resolution to establish \cref{introthm:golod}, which provides an exponential lower bound on the number of $k$-summands of syzygies of $M$, rather than the constant bound provided by \cref{introthm:general}.
We also establish this lower bound in one lower homological degree than that of \cref{introthm:general}.
In this section $\vp\colon Q \to R$ will be a local (or graded-local) homomorphism of Burch index at least two, and $M$ will be a $\vp$-Golod module, defined in the sequel.

\begin{chunk}\label{c:golodproperties}
The Golod property is characterized by minimality of a bar resolution $B(R,X,Y)$ of $M$.
Indeed, even when $M$ is not Golod, $B(R,X,Y)$ can still be constructed using $A_\infty$ structures on the minimal resolutions $X$ of $R$ and $Y$ of $M$ \cite{Bur-15,Iye-97}, as described in \cref{chu-bar}.
The resulting resolution $B(R,X,Y)$ of $M$ provides an upper bound for the growth of the Betti numbers of $M$. 
In particular, examining the relationship between the generating function of the ranks of $B(R,X,Y)$ with respect to those of $X$ and $Y$ reveals the coefficient-wise inequality of Poincar\'e series
\[P_M^R(t)\preccurlyeq \frac{P_M^Q(t)}{1-t(P_R^Q(t)-1)}\] 
and the classical definition of the Golod property is that equality holds in the above \cite{Avramov:2010}, demonstrating that the Betti numbers of Golod modules attain maximal growth.
In particular, one sees that for any $A_\infty$ algebra structure on the minimal $Q$-free resolution $X$ of $R$ and any $A_\infty$ $X$-module structure on the minimal  $Q$-free resolution $Y$ of $M$, the products $m_n\colon X^{\otimes n}\rightarrow X$ and $\mu_n\colon X^{\otimes n-1}\otimes Y\rightarrow Y$ are all minimal \cite[6.3]{Bur-15}, that is,

\[m_n(X_+\otimes\dots\otimes X_+)\subset \n X\,\,\,\text{and}\,\,\,\mu_n(X_+\otimes\dots\otimes X_+\otimes Y)\subset \n Y.\]

An important case is when $k$ is $\vp$-Golod with respect to the minimal Cohen presentation of $R$, in which case $R$ is said to be Golod.
Similarly, if $M$ is $\vp$-Golod with respect to the minimal Cohen presentation of $R$, then $M$ is said to be Golod.
The Golod property then has profound implications for the homological and representation theoretic properties of $R$; in particular, if $R$ is Golod and not a hypersurface, than the Betti sequence of any module of infinite projective dimension must grow exponentially.

Despite the apparent specificity of Golod properties, all rings and modules are ``asymptotically Golod'' in the following senses: if $I$ is an ideal in a regular local ring $Q$, then $Q/I^k$ is Golod for $k\gg 0$ \cite{HerWelYas-2016}.
More generally, when $(Q,\n)$ is only assumed to be local, the map $Q\rightarrow Q/\n^n$ is Golod for $n\gg0$ \cite{Lev-75}.
If $M$ is a module over a Golod ring, then its syzygies $\syz_i(M)$ are Golod for $i\geq \edim(R)-\depth(M)$\cite{Les-90}.
One way to obtain Golod rings of high Burch index is to take a quotient of a regular or graded ring $(Q,\n)$ by an ideal of the form $J\n$, where $J$ is a an ideal of $Q$.
For such rings, the Burch index is $\dim(Q)$ \cite[1.5]{DaoEis-23}.
\end{chunk}

The following establishes \cref{introthm:golod}.

\begin{theorem}\label{thm:golod}
Let $\vp\colon (Q,\n,k)\to (R,\m,k)$ be a surjective local (or graded-local) homomorphism with $\Burch(\vp)=b\geq 2$, and let $M$ be a finitely generated (resp., finitely generated graded) $R$-module such that $M$ is $\vp$-Golod.
Then for $q\geq 3$,
\[
\krank({\syz_{q+1}^RM})\geq {b\choose 2} m^{\lfloor{\frac{q-3}{2}}\rfloor}.
\]
where $m=\mu(\ker\vp)$ is the minimal number of generators of $\ker \vp$. 
\end{theorem}

\begin{proof}
By \cref{prop:burkeAinfty}, the minimal $Q$-free resolution $X$ of $R$ admits the structure of an $A_\infty$ algebra and the minimal $Q$-free resolution $Y$ of $M$ admits the structure of an $A_\infty$ $X$-module.
When $M$ is $\vp$-Golod, the resulting bar resolution $B(R,X,Y)$ as described in \cref{chu-bar} is minimal as discussed in \cref{c:golodproperties}.

Set $I=\ker \vp$. As discussed in \cref{dfn-burch}, one may choose elements $x_1,\dots,x_n$ that minimally generate $\n$ ordered so that the images of $x_1,\dots,x_b$ are linearly independent in $\n/\BI(\vp)$.
As determined in \cref{chu-fij} and \cref{lem-independent}, these yield independent basis elements $f_{x_i,x_j}$ in $X_2$ for $1\leq i < j\leq b$ with the property that 
\[ \del(f_{x_i,x_j})\in \n X_1\setminus \n IX_1.\]
Further, noting that $m=\mu(I)=\rank_Q X_1$, take any basis $\{ e_1,\dots,e_m \}$ for for $X_1$. 
Write $q-3=2d+r$ where $d\geq 0$ and $r\in\{0,1\}$, and choose a basis element $y\in Y_r$.

Then we claim that the following elements of $B=B(R,X,Y)$ in degree $q$
\[
s_i [f_{x_j,x_i} | e_{i_1} | \cdots | e_{i_d}] y
\in B_q
\qquad 
1 \leq i < j \leq b 
{\textrm{ and }}
1 \leq i_k \leq m
\]
are independent cycles generating a $k$-vector space summand of $\del^B_q$. 
	Indeed they are cycles as the assumptions ensure that the differential of $B(R,X,Y)$ is minimal over $R$ and hence its image is killed by $s_i\in\soc R$. 
By \cref{lem-independent}, the set $\{ f_{x_j,x_i} | 1\leq i<j\leq b \}$ is independent in $X_2$ and hence the elements $s_if_{x_j,x_i}$ for $1\leq i<j\leq b$ generate a $k$-vector space in $R\otimes X_2$. 
Lastly, that these cycles are a summand of the $q$-th syzygy follows from \cref{lem-splitting}: 
The elements 
\[
\rho_{i,j}=[f_{x_j,x_i} | e_{i_1} | \cdots | e_{i_r}] y
\]
are part of a basis of $B_q$.
Lifting the differential $\del\colon B_q\rightarrow B_{q-1}$ to $Q$ and $\rho_{i,j}$ to a basis element $\rho=\widetilde{\rho_{i,j}}$, the assumptions on $f_{x_j,x_i}$ ensure that $\del(\rho)\notin \BI(\vp)B_{q-1}$. 
Hence each of these cycles generate a $k$-summand of the $q$-th syzygy by \cref{lem-splitting}, which suffices to yield the result. 
\end{proof}

\begin{corollary}\label{cor:Rgolod}
Let $R$ be a Golod ring and $M$ be an $R$-module. Then the exponential growth of $\krank({\syz_{i}^RM})$ begins by $i=\min\{9,\edim(R)-\depth(M)+4\}$.
\end{corollary}

\begin{proof}
By the result mentioned in \cref{c:golodproperties}, the $R$-module $\syz_i^R(M)$ is Golod if $i\geq\edim(R)-\depth(M)$. 
On the other hand, by \cref{thm:general}, $\syz_i^R(M)$ has a Golod direct summand (namely, $k$) for $i\geq 5$. 
In either case, \cref{thm:golod} insures that the Golod module will contribute exponential growth of the $\krank$ by at most 4 steps later. 
\end{proof}

\begin{remark}
\label{rmk:exponential}
Lastly we note that, even when $R$ is not Golod, it already follows from \cite{DaoEis-23} that a \emph{subsequence} of $\krank(\syz_R^i(M))$ grows exponentially whenever $\Burch(R)\geq 2$. Indeed, note that $\krank(\syz_R^5(M))\geq 1$, and so the resolution of $M$ contains a direct summand of the resolution of $k$, Therefore, it suffices to show this fact for $M=k$. We have that $\syz_R^5(k)\isom k\oplus M'$, and we claim that $M'\neq 0$: Since $\edim(R)\geq \Burch(R)\geq 2$, $R$ is not the hypersurface $k[x]/(x^2)$, so by examining the acyclic closure \cite{Tate:1957}, we get $M'\neq 0$.
Then we get in turn that $\krank(\syz_R^5(M')\geq 1$, and so 
\[\krank(\syz_R^{10}(k))\geq \krank(\syz_R^5(k))+\krank(\syz_R^5(M')\geq 2,\]
Continuing inductively, we see $\krank(\syz_R^{10n})\geq 2^n$.
\end{remark}

\printbibliography

\end{document}